\documentclass{article}
\usepackage{spconf,amsmath,graphicx}

\usepackage[ruled]{algorithm2e}
\usepackage{amsthm}
\usepackage{float}
\usepackage{amsmath,graphicx}
\usepackage{amssymb}
\newtheorem{theorem}{Theorem}
\newtheorem{lemma}{Lemma}
\newtheorem{definition}{Definition}

 \usepackage{hyperref}
\usepackage{color}
\definecolor{brightmaroon}{rgb}{0.76, 0.13, 0.28}
\definecolor{prettygreen}{rgb}{0.13,0.76 , 0.28}
\hypersetup{
    colorlinks=true, 
    linkcolor=brightmaroon,  
    citecolor=prettygreen,  
    filecolor=brightmaroon,  
    urlcolor=brightmaroon,   
}
\title{Exploiting the Structure via Sketched Gradient Algorithms}
\name{Junqi Tang, Mohammad Golbabaee, Mike Davies}
\address{Institute for Digital Communications,
University of Edinburgh.
EH9 3JL,
UK}

\begin{document}
\ninept
\maketitle
\begin{abstract} 

%{\it Sketched gradient methods} \cite{tang2016gradient} have been previously introduced for efficiently solving the large-scale constrained Least-squares regressions. Those sketched gradient methods utilizes iterative sketching techniques to reduce the sample size w.r.t the intrinsic low-dimensional structure of the constrained set where the solution lives in, and is proven to be able to exploit this structure to accelerate computation. In this paper we reveal the trade-off between the sketch size and step size of the sketched gradients and computational cost w.r.t the (sparsity) constraint, via theoretical analysis and numerical experiments on large data sets and image reconstruction applications.

{\it Sketched gradient algorithms} \cite{tang2016gradient} have been recently introduced for efficiently solving the large-scale constrained Least-squares regressions. In this paper we provide novel convergence analysis for the basic method {\it Gradient Projection Classical Sketch} (GPCS) to reveal the fast linear convergence rate of GPCS thanks to the intrinsic low-dimensional geometric structure of the solution prompted by constraint set. Similar to our analysis we observe computational and sketch size trade-offs in numerical experiments. Hence we justify that the combination of gradient methods and the sketching technique is a way of designing efficient algorithms which can actively exploit the low-dimensional structure to accelerate computation in large scale data regression and signal processing applications. 

\end{abstract} 

\section{Introduction}

%We are now in an era of boosting knowledge and large data. In our daily life we have various signal processing and machine learning applications which involve the problem of tackling a huge amount of data. These applications vary from Empirical Risk Minimization (ERM) for statistical inference, to medical imaging such as the Computed Tomography (CT) and Magnetic Resonance Imaging (MRI), channel estimation and adaptive filtering in communications, and in machine learning problems where we need to train a neural network or a classifier from a large amount of data samples or images. Many of these applications involve solving constrained optimization problems. In a large data setting a desirable algorithm should be able to simultaneously address high accuracy of the solutions, small amount of computations and high speed data storage.

In many signal processing, data mining and machine learning applications, people encounter large scale optimization tasks which involve regression on a huge amount of data. In recent years there has been a growth in the amount of research on fast first order optimization algorithms such as the variance-reduced stochastic gradient descent \cite{2013_Johnson_Accelerating} \cite{defazio2014saga} which has scalability with respect to the amount of the data points to be processed by constructing a stochastic estimate of the truth gradient each iteration. Such novel algorithms has been shown to be effective in various large scale regression problems when compared to conventional full gradient methods such as FISTA \cite{beck2009fast}.

Classical randomized sketching techniques  \cite{clarkson2013low}\cite{2015_Pilanci_Randomized} and iterative sketching techniques \cite{2015_Pilanci_Newton}\cite{2016_Pilanci_Iterative} have been recently introduced to construct subprograms with a reduced sample-size. Meanwhile the maximum sample size reduction (aka. the smallest sketch size) is directly related to the intrinsic dimension of the desired solution prompted by low-dimensional structure such as sparsity, group sparsity and low-rank. Such sketching techniques open the door for designing optimization algorithms which are able to exploit these low-dimensional structures and hence achieve scalability and acceleration. Very recently, the sketched gradient algorithms \cite{tang2016gradient} have been proposed based on a combination of iterative sketching and projected gradient descent for constrained Least-squares regression. The sketched gradient algorithms merit the strength of both full gradient algorithms and stochastic gradient algorithms and appear to be efficient in practice.

This paper focuses on exploring how to optimally exploit the constraint set via randomized sketched gradient algorithms to speedily and approximately solving large scale constrained Least-square estimation problems. 

\section{Background}
%Consider a noisy linear measurement model for a vector $x_{gt}$ (ground truth) which belongs to a convex constrained set $\mathcal{K}$, an $n$ by $d$ linear operator matrix $A$ with $n > d$, and additive noise denoted by $w\in \mathcal{R}^{n \times 1}$:

In many signal processing and image reconstruction applications the forward model is linear, which involves the ''ground truth'' $x^\star$ which a priori is assumed to belong to model (constraint) set $\mathcal{K}$, with measurement operator $A$ and the noise vector $w$:
\begin{equation}\label{eq:1}
{y} = {A} {x}^\star + {w}, \ \ \ \  \ {x}^\star \in \mathcal{K}, \ \ \ {A} \in \mathbb{R}^{n \times d}.
\end{equation}
In this paper we focus on the large data setting $n >> d$. We wish to estimate $x^\star$ via a constrained Least-squares (LS) estimator:
%In the context of imaging applications such as X-ray Computed Tomography (CT), the vector $y$ denotes a set of $n$ physical measurements collected from an image $x_{gt}$ through the measurement operator $A$. In the context of machine learning, $A$ is often a training data matrix used for setting the regression parameters $x^\star$ from the observations $y$. The constrained Least-square (LS) estimator for $x^\star$ is:	
	\begin{equation}\label{LS}
    x_{LS}=\arg\min_{x\in \mathcal{K}} \left\{f(x) := \|y-Ax\|^2_2 \right\}.
    \end{equation}
A popular approach for (\ref{LS}) is the Projected Gradient Descent (PGD) \cite{figueiredo2007gradient}:
\begin{equation} \label{eq:4}
    x_{j+1}=\mathcal{P}_\mathcal{K}(x_{j}-\eta A^T(Ax_{j}-y)).
\end{equation}
The PGD/ISTA algorithm and its accelerated variant FISTA \cite{beck2009fast} have been of great interest in signal processing and compressed sensing applictions since they cope with various type of signal models prompted by hard constraints or non-smooth regularization. We refer to these methods as {\it full gradient methods} (FG). The main drawback of FG is when the linear systems' size is large, each iteration of FG methods becomes costly to compute. Alternatively, recent works on sketching techniques \cite{2011_Mahoney_Randomized} \cite{drineas2011faster}\cite{2015_Pilanci_Randomized}\cite{2016_Pilanci_Iterative} and consequently the sketched gradient methods \cite{tang2016gradient} which was built on both the {\it Classical Sketch} \cite{2015_Pilanci_Randomized} as the fast initialization step and {\it Iterative Hessian Sketch} \cite{2016_Pilanci_Iterative} for solving the Least-squares regression to machine precision. In this paper we focus particularly on the novel convergence analysis of the initialization step which is called {\it Gradient Projection Classical Sketch}.

\subsection{Main contributions}

\begin{itemize}
    \item For the sketched Least-square problem where very often we may wish to have a sketch size $m$ which is smaller than the ambient dimension $d$, from an optimization point of view, it is unknown that when a first order method can be efficiently used and whether the convergence speed can be fast since the global strong convexity assumption is vacuous in this setting. We show that if a Gaussian sketching matrix is used to construct the sketched problem and we choose to use the standard projected gradient descent as the solver, this method converges towards the vicinity of $x^\star$ with a linear convergence rate, as long as the sketch size $m$ is larger than a certain measure (Gaussian width) of the intrinsic dimension of $x^\star$ associated with the constraint set $\mathcal{K}$. Moreover, the convergence speed is also a function of the sketch size choice and the statistical dimension. Hence we demonstrate the fact that by combining the projected gradient methods and randomized sketching techniques, one can design accelerated algorithms which can {\it actively} exploit the statistical structure of the signal/image to be estimated. While in previous work \cite{tang2016gradient}, the GPCS is also analyzed but in a conservative way which does not include the case of choosing a sketch size smaller than the ambient dimension.
    \item We study the computational trade-off between the sketch size and the computational cost with respect to the structure of the ground truth and the constrained set for the sketched problem solved by the projected gradient descent. 
\end{itemize}

\section{Gradient projection Classical sketch}

We first introduce the {\it Gradient Projection Classical Sketch} (GPCS) which is the initialization loop of our previously proposed {\it Gradient Projection Iterative Sketch} algorithm. This simple algorithm is basically running the PGD to solve the sketched Least-square problem (\ref{CS}). We denote $S \in \mathbb{R}^{m \times n} (m \ll n)$ as the randomized sketching operator which significantly reduces the sample dimension of the LS estimator.
	\begin{equation}\label{CS}
    \hat{x}=\arg\min_{x\in \mathcal{K}}\left\{f_0(x) := \|Sy-SAx\|^2_2\right\},
    \end{equation}
In next section, we describe the convergence behavior of the GPCS algorithm.

\begin{algorithm}\label{A1}
\SetAlgoLined

 Initialization: $x_0 = 0$\;
 Given $A \in \mathbb{R}^{n \times d}$, sketch size $m \ll n$\;
 Prior knowledge: the true solution $x$ belongs to set $\mathcal{K}$ \;
  Generate a random sketching matrix $S \in \mathcal{{R}}^{m \times n}$\;
  Calculate $SA$, $Sy$\;

 \While{$i=0:k - 1$}{

       $x_{i+1}=\mathcal{P}_\mathcal{K}(x_{i}-\eta (SA)^T(SAx_{i}-Sy))$\;

 }
Return $x = x_{k}$\;

 \caption{Gradient Projection Classical Sketch}
\end{algorithm}

\section{Convergence analysis}

We start by defining some certain properties of the operator $A$, sketching scheme $S$, and its interaction with the signal model (constraint set), in a similar manner to the analysis in \cite{tang2016gradient}\cite{2015_Oymak_Sharp} and \cite{2015_Pilanci_Randomized}: %Throughout this part we denote by $\mathcal{C}$ be the smallest closed cone at $x^\star$ containing the set $\mathcal{K}-x^\star$.
\begin{definition}\label{D1}
Let $\mathcal{C}$ be the smallest closed cone at $x^\star$ containing the set $\mathcal{K}-x^\star$, 
$\mathbb{S}^{d-1}$ be the unit sphere in $\mathcal{R}^d$,
$\mathcal{B}^{d}$ be the unit ball in $\mathbb{R}^d$, $z$ be an arbitrary fixed unit-norm vector in $\mathbb{R}^n$. The contraction factor $\alpha(\eta,SA)$, the error amplification factor $\beta(S,A)$ are defined as:
\begin{align}
&\alpha(\eta,SA)=\sup_{u,v \in \mathcal{C} \cap \mathcal{B}^d} v^T(I-\eta A^TS^{T}SA)u,\\
&\beta(S,A)=\sup_{v \in A\mathcal{C} \cap \mathcal{B}^n} v^T\frac{S^{T}S}{m}z.
\end{align}
\end{definition}

For the convenience of the presentation of our main theorem, we shall denote these two key factors defined in Def.\ref{D1} as $\alpha := \alpha(\eta,SA)$ and $\beta := \beta(S,A)$, respectively.

\begin{definition}\label{D2}
The cone-restricted strong convexity constant $\mu_c$ is defined as the largest positive constant which satisfies: for all $z_c \in \mathcal{C}$
\begin{equation}
    \|Az_c\|_2^2 \geq \mu_c \|z_c\|_2^2.
\end{equation}
\end{definition}

The cone restricted strong-convexity defined here can be viewed as a recovery/inversion stability measure which ensures that the original Least-squares estimator is reliable and robust to noise: $\|x_{LS} - x^\star\|_2 \leq \frac{2\|w\|_2}{\mu_c}$, see \cite[Proposition 2.2]{2012_Chandrasekaran_Convex}.

\begin{definition}\label{D3}
The Lipschitz constant $L$ for the LS (\ref{LS}) is defined as the largest singular value of the Hessian matrix $A^TA$: for all $z_d \in \mathbb{R}^d$
\begin{equation}
    \|Az_d\|_2^2 \leq L \|z_d\|_2^2,
\end{equation}
\end{definition}

Now we are ready to present our main result on linear convergence of the GPCS:
\begin{theorem}\label{T1}
   Starting from $x_0$, if the step size $\eta$, the sketching operator $S \in \mathbb{R}^{m \times n}$ and sketch size $m$ are properly chosen such that $\alpha(\eta,S^tA) < 1$, the following error bound holds:
   \begin{equation}
   \begin{aligned}
        \|x_k-x^\star\|_2 \leq \alpha^{k}\|x_0-x^\star\|_2 +   \frac{m\eta L\beta}{1-\alpha}\|w\|_2,       
   \end{aligned}
   \end{equation}
\end{theorem}
\begin{proof}
The proof of Theorem \ref{T1} can be found in appendix.
\end{proof}
Theorem \ref{T1} reveals that as long as the step size, the sketching matrix and sketch size are chosen properly, the updates sequence $x_1, x_2, ..., x_k $ generated by the GPCS algorithm converge {\it exponentially}\footnote{In optimization community, this convergence rate also referred to as linear convergence} towards the vicinity of the ground truth $x^\star$ with a distance scales with the noise energy $\|w\|_2$. In the forthcoming subsection we explicitly quantify the two pillar factors $\alpha$ and $\beta$ in Theorem \ref{T1} when $S$ is a Gaussian sketch, and hence reveals the structure-exploiting property of the GPCS algorithm.

\subsection{Explicit analysis results for Gaussian sketches}	
Theorem 1 has provided us a general framework to describe the convergence of GPCS in terms of $\alpha$ and $\beta$. We can derive expressions of
these terms in terms of the Gaussian Width of the set  $\mathcal{W} := \mathcal{W}(A\mathcal{C}\cap\mathbb{S}^{n-1})$ and the sketch dimension $m$ when we choose $S$ to be a Gaussian sketching matrix. The Gaussian Width is defined as:

\begin{definition}\cite[Definition 3.1]{2012_Chandrasekaran_Convex} \label{D4}
The Gaussian width of a convex set $\Omega$ is defined as:
\begin{equation}
    \mathcal{W}(\Omega) = E_g\left( \sup_{v \in \Omega} v^Tg\right),
\end{equation}
where $g \in \mathcal{R}^n$ is draw from distribution $\mathcal{N}(0,I_n)$.
\end{definition}
The square of Gaussian width $\mathcal{W}^2(\mathcal{C}\cap\mathcal{S}^{d-1})$ is a well-known tool of measuring the statistical intrinsic dimension of $x^\star$ enforced by the constrained set $\mathcal{K}$. If $x^\star$ has only $s$ non-zero entries, and we construct the Least-square estimator with a $l_1$ constraint, the width $\mathcal{W}^2(\mathcal{C}\cap\mathcal{S}^{d-1})$ can be upper bounded by $2slog(\frac{d}{s}) + \frac{5}{4}s$ \cite[Proposition 3.10]{2012_Chandrasekaran_Convex}, which means $\mathcal{W}(\mathcal{C}\cap\mathcal{S}^{d-1})$ scales with the sparsity of $x^\star$.

%The value of $\mathcal{W}(\mathcal{C}\cap\mathcal{S}^{d-1})$ is a very useful measurement of the tightness of the structure of $x^\star$, for example, if $x^\star$ is $s$-sparse and we model the sparsity constraint using an $l_1$ ball, we will have $\mathcal{W}(\mathcal{C}\cap\mathcal{S}^{d-1}) \leq \sqrt{2slog(\frac{d}{s}) + \frac{5}{4}s}$, which means the sparser $x^\star$ is, the smaller the $\mathcal{W}(\mathcal{C}\cap\mathcal{S}^{d-1})$ will be \cite[Proposition 3.10]{2012_Chandrasekaran_Convex}. As an illustration we now quantify the bounds in Theorem 1 in terms of the sketch size $m$ and the defined Gaussian width.

In parallel to the step size choices described in \cite{2015_Oymak_Sharp} (which is about the PGD convergence analysis on solving the original LS problems where $A$ is a Gaussian map),  our analysis covers a greedy choice and a conservative choice of the step size $\eta$. We denote $b_m=\sqrt{2}\frac{\Gamma(\frac{m+1}{2})}{\Gamma(\frac{m}{2})}\approx\sqrt{m}$ as the same in \cite{tang2016gradient} and \cite{2015_Oymak_Sharp}. For the greedy choice, we can bound $\alpha$ as described in Lemma \ref{L1}.

\begin{lemma}\label{L1}
(Greedy step size) If the step-size $\eta = \frac{1}{b_m^2L}$, and  the entries of the sketching matrix $S$ are i.i.d drawn from Normal distribution, then:
\begin{equation}
    \alpha(\eta, SA) \leq (1 - \frac{\mu_c}{L})(1 + \frac{(\mathcal{W} + \theta)^2}{b_m^2}) + \frac{\sqrt{8}(\mathcal{W} + \theta)}{b_m},
\end{equation}
with probability at least $1 - 8e^{-\frac{\theta^2}{8}}$.
\end{lemma}

This result for the greedy step size is reminiscent of the result for greedy step size described in \cite[Theorem 2.2]{2015_Oymak_Sharp}, in a sense that when $\frac{L}{\mu_c} \to 1$, we recover their convergence result. However, Lemma 1 does not ensure $\alpha < 1$ since it inherently demands a sketch size $m \gtrsim \left(\mathcal{W}\frac{L}{\mu_c}\right)^2$. This sketchsize requirement can be moderated, at the cost of a more conservative stepsize:

\begin{lemma}\label{L2}
(Conservative step size) If the step-size $\eta = \frac{1}{L(b_m + \sqrt{d} + \theta)^2}$, and  the entries of the sketching matrix $S$ are i.i.d drawn from Normal distribution, then:
\begin{equation}
    \alpha(\eta, SA) \leq \left\{1-\frac{\mu_c}{L}\frac{(b_m-\mathcal{W}- \theta)^2}{(b_m+\sqrt{d}+ \theta)^2}\right\}.
\end{equation}
with probability at least $(1-2e^{-\frac{\theta^2}{2}})$.
\end{lemma}

From Lemma \ref{L2} we can see that when the conservative step size is used, the sketch size we need to ensure $\alpha < 1$ is only $m \gtrsim \mathcal{W}^2$.

\begin{lemma}\label{L3}
(Bound on noise amplification factor $\beta$) If the entries of the sketching matrix $S$ are i.i.d drawn from Normal distribution, then:
\begin{equation}
    \beta \leq 1 + \frac{\sqrt{2}b_m(\mathcal{W} + \theta)}{m} + |\frac{b_m^2}{m} - 1|,
\end{equation}
with probability at least $(1-2e^{-\frac{\theta^2}{2}})$.
\end{lemma}
\begin{proof}
The proofs of Lemma \ref{L1}, \ref{L2} and \ref{L3} can be found in appendix.
\end{proof}
Lemma \ref{L2} and \ref{L3} reveals that the sketch size $m$ has an impact on both the convergence speed and the noise amplification. For the convergence speed, the larger the sketch size $m$ w.r.t the Gaussian width $\mathcal{W}$ is, the steeper this convergence can be, but on the other hand we should not choose the sketch size to be too large since each iteration will become more expensive to compute hence there exists a trade-off between sketch size and computation. The noise amplification is a decreasing function w.r.t the sketch size $m$, which means the larger sketch size we choose, the accuracy of the GPCS output will increase. We next explore these theoretical findings through numerical experiments.

\section{Numerical experiments}
\subsection{Synthetic constrained LS estimation example}
We first test the behaviour of GPCS algorithm with various sketch sizes for sparse linear inversion task which recovers a sparse vector $x^\star$ from noiseless measurements $y = Ax$. The $l_1$ norm of $x^\star$ is assumed known as a prior hence hence we construct the constraint set $\mathcal{K} = \left\{ \forall v : \|v\|_1 \leq \|x^\star\|_1\right\}$.

The experiments are executed on a PC (2.6 GHz CPU, 1.6 GB RAM) with MATLAB R2015b.

The details of the experimental setting can be found in Table 1, and meanwhile the procedure of generating the synthetic data matrix A with a condition number $\kappa$ follows \cite{tang2016gradient}:

1) Generate a random matrix $A$ sized $n$ by $d$ using MATLAB command Randn.

2) Compute the Singular Value Decomposition of $A$: $A = U\Sigma V^T$ and modify the singular values $e_i = diag(\Sigma)_i$ by:
\begin{equation}
    e_{i} = \frac{e_{i - 1}}{\kappa^{\frac{1}{d}}},
\end{equation}
to achieve the condition number of $\frac{L}{\mu} = 10^6$.

The performance of the Projected Gradient and fast gradient method (FISTA) for SYN-SLI is shown in figure 1. The step sizes for all the algorithms are generated by the line-search given by \cite{nesterov2007gradient}. Although in this paper we analyze the explicit convergence speed and noise amplification for the Gaussian Sketch as a motivational theory, this type of sketches in costly to compute. In practice, instead of using directly the Gaussian sketch, people use faster sketches such as the Fast Johnson-Lindenstrauss Transform (FJLT) \cite{2008_Ailon_Fast}\cite{2009_Ailon_Fast}, the sparse JLT \cite{2014_Kane_Sparser} and the Count Sketch \cite{clarkson2013low}. We choose to use the Count-Sketch \cite{clarkson2013low} for the GPCS to speedily produce the sketched matrix $SA$.

This experimental result confirms that for noiseless inversion $y = Ax$, the GPCS converges towards the ground truth with an exponential rate as our theory predicts, and also the best convergence is given by appropriate median sketch size choices ($m = 700, 1100$), which are some factors larger than the sparsity $s = 50$, but less than the ambient dimension $d = 1500$. The GPCS provides significant computational benefits over the full gradient methods PGD and FISTA on this large scale inversion task.

We then turn to the noisy set up for SYN-SLI $y = Ax + w$ (Fig. 2). The Gaussian noise vector $w$ satisfies $\frac{\|Ax^\star\|_2}{\|w\|_2} = 30$. Here we find out that as our theory predicted, the GPCS converges to a vicinity of the ground truth $x^\star$, meanwhile the larger the sketch size is, the more accurate the solution is. That is, the GPCS converge to an approximated solution of the Least-squares, and the approximation accuracy is determined by the noise level and the sketch size $m$. As discussed in \cite{tang2016gradient}, if one needs to solve the LS to machine precision in the presence of noise $\|w\|_2 > 0$, one can use the GPCS as a fast initialization step and then continue by GPIHS which is based on the ''iterative'' sketches \cite{2016_Pilanci_Iterative}.
\begin{table}[t]\label{Syn}
\caption{Experiment settings. (*) s denotes sparsity of the ground truth $x^\star$}
\label{sample-table}
\vskip 0.15in
\begin{center}
\begin{small}
\begin{sc}
\begin{tabular}{lcccr}
\hline
%\abovespace\belowspace
Data set & Size & (*)$s$ & $\frac{L}{\mu}$\\
\hline
%\abovespace
Syn-SLI & (100000, 1500) & 50 & $10^6$ \\
Syn1    & (10000, 100) & 5 & $10^6$ \\
Syn2    & (100000, 100) & 5 & $10^6$ \\
%\belowspace
Syn3      & (200000, 100) & 5 & $10^6$  \\
\hline
\end{tabular}
\end{sc}
\end{small}
\end{center}
\vskip -0.1in
\end{table}
	\begin{figure}[t] %\label{fig:3} %{\textwidth}
	\centering	
         \includegraphics[width=80mm]{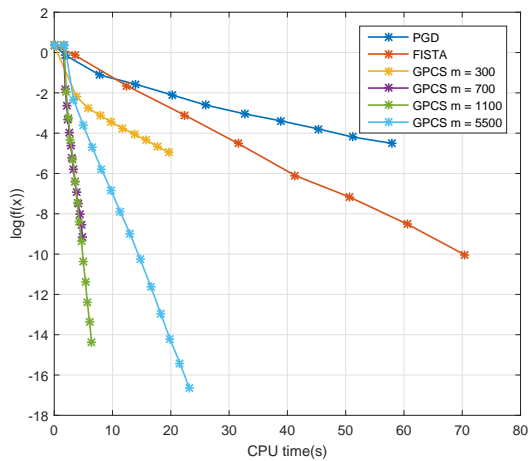}
         	\caption{ Large scale noiseless sparse linear inversion experiment (Syn-SLI), $y = Ax$  }
	\end{figure} %\vspace{-1cm} 
	\begin{figure}[t] %\label{fig:3} %{\textwidth}
	\centering	
         \includegraphics[width=80mm]{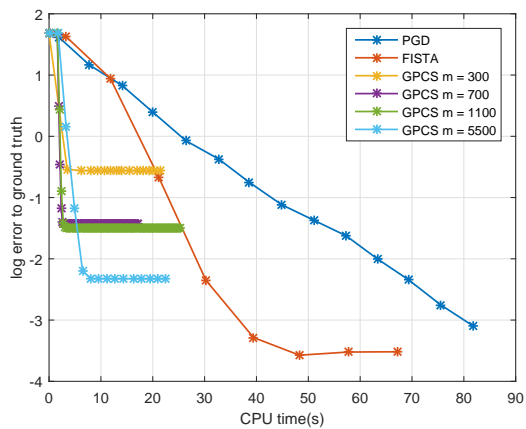}
         	\caption{ Large scale noisy Least-squares regression experiment (Syn-SLI), $y = Ax + w$ }
	\end{figure} %\vspace{-1cm} 	

\subsection{Computational and sketch size trade-off}

We finally examine the computational cost of GPCS on the linear inversion task $y = Ax$ through different choices of sketch size on three synthetic examples (Syn1, Syn2 and Syn3). Here we test GPCS on both the Gaussian Sketch and the Count-Sketch which is more computationally efficient. We run GPCS with sketch size from 10 to 1000 until a fixed accuracy $\|x^t - x^\star\|_2 \leq 10^{-4}$ is achieved, or when the computational budget has run out. We average the results of 20 random trials.

From the experimental results (Fig. 3) we can observe a sharp trade-off phenomenon as our theory predicted. When the sketch size is too near to the intrinsic dimension, the GPCS takes a huge cost to achieve the targeted accuracy; but if we increase the sketch size by a small amount, the computational cost drops radically to a ''sweet spot'', then if we continue to increase the sketch size, the cost increases again since each iteration's cost is more expensive for a large sketch size. We can also observe that the optimal choice of sketch size is not a function of the data sample size $n$, as our theory predicted. Surprisingly we see that the larger the $n$ is, the behavior of the Gaussian Sketch and Count Sketch become more similar as shown in the right hand figures in Fig 3. It would be an interesting future research direction to investigate in theory the relationship between the performance of the practical sparse embedding schemes such as the Count Sketch, and the properties of $A$, e.g. the sample size $n$, the parameter dimension $d$ and the conditioning $\frac{L}{\mu}$, the distribution of its singular values and singular vectors, etc.

	\begin{figure}[t] %\label{fig:3} %{\textwidth}
	\centering	
		\includegraphics[width=90mm]{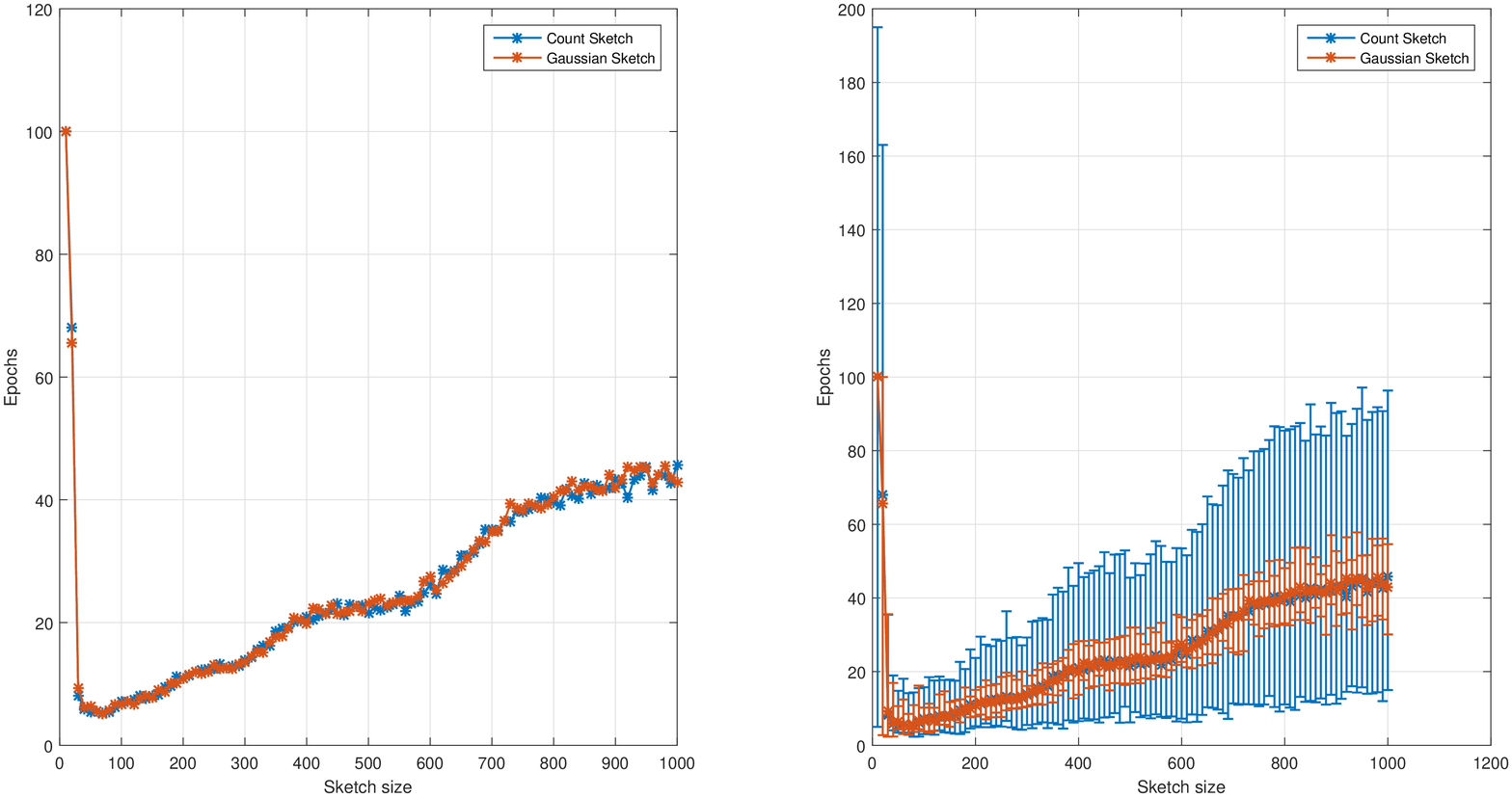}
		\includegraphics[width=90mm]{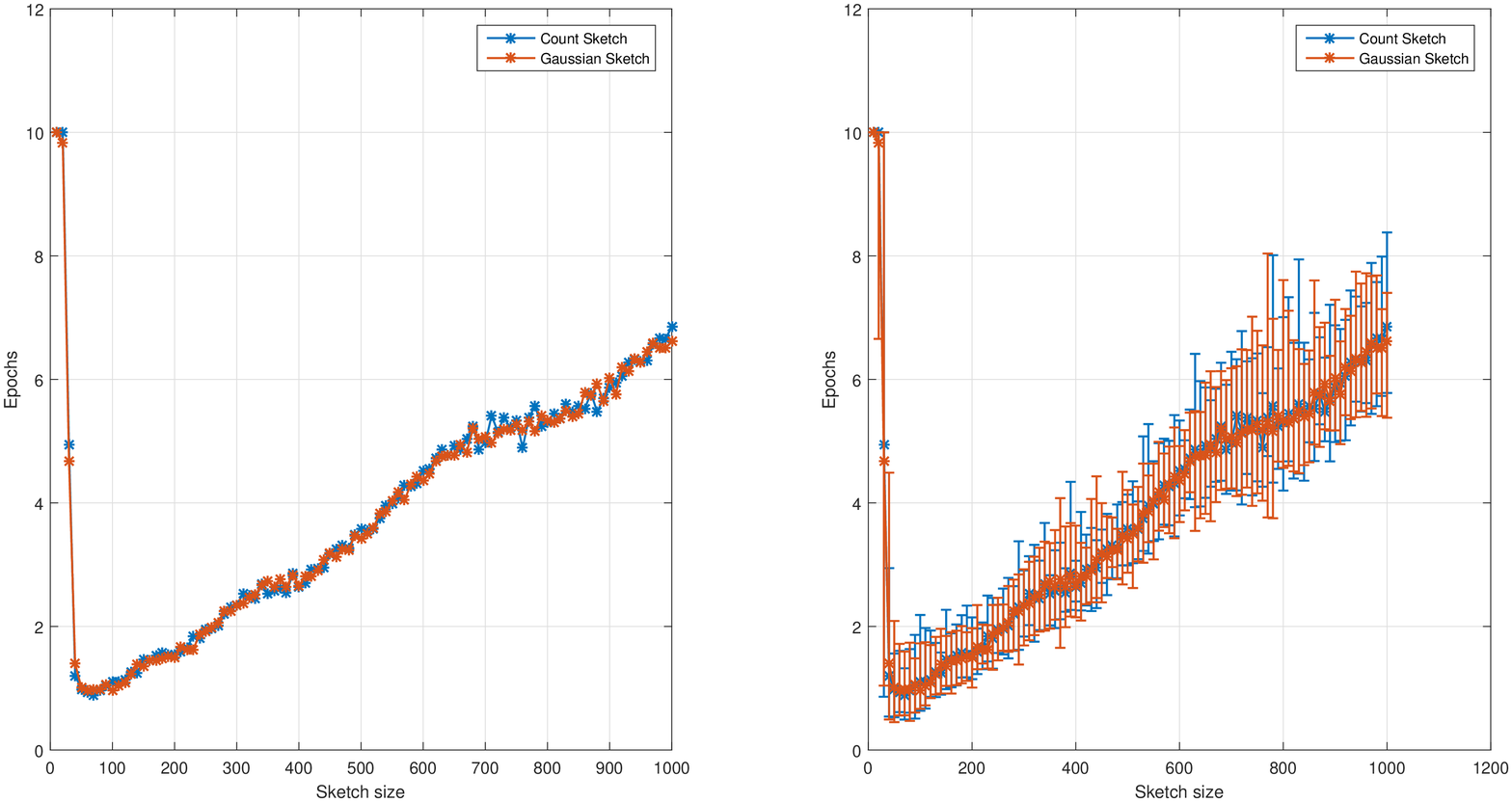}
		\includegraphics[width=90mm]{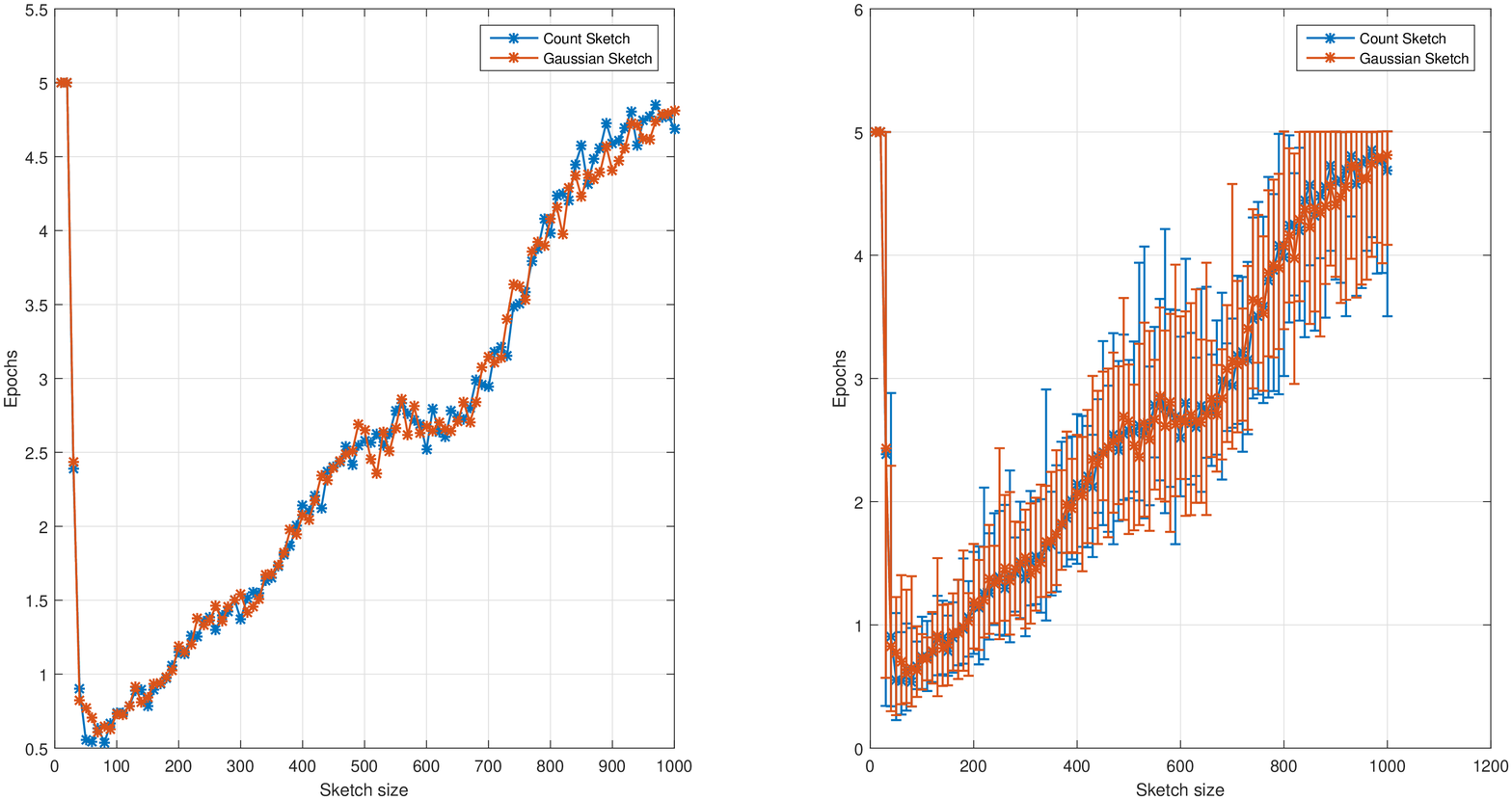}
	\caption{ Computational and sketch size trade-off experiments for GPCS on $l_1$ constrained linear system $y = Ax$. From top to bottom : Syn1, Syn2 and Syn3; from left to right: average curve (on 20 trials) and error bar (the maximum and minimum epoch counts in these 20 trials)}
	\end{figure} %\vspace{-1cm} 

\section{Conclusion}

We prove for the first time that when the original constrained Least-square problem satisfies the robust recovery condition, the exponential convergence holds for the GPCS algorithm towards a vicinity of the ground truth without any assumption on the strong convexity of the sketched Least-squares, thanks to the concentration results for the Gaussian sketch and the intrinsic dimension of the ground truth. The numerical experiments vindicate our theoretical findings such as the convergence speed and noise amplification of GPCS algorithm. We also demonstrate the huge potential of the sketched gradient method based on exploiting the low dimensional structure (such as sparsity) via sketching techniques, by comparing with the well known full gradient methods PGD and FISTA.

\section{acknowledgements}
JT, MG and MD would like to acknowledge the support from H2020-MSCA-ITN Machine Sensing Training Network (MacSeNet), project 642685; EPSRC Compressed Quantitative MRI grant, number EP/M019802/1; and ERC Advanced grant, project 694888, C-SENSE, respectively. MD is also supported by a Royal Society Wolfson Research Merit Award.
%\section{acknowledgements}

% Below is an example of how to insert images. Delete the ``\vspace'' line,
% uncomment the preceding line ``\centerline...'' and replace ``imageX.ps''
% with a suitable PostScript file name.
% -------------------------------------------------------------------------

%\vfill
%\pagebreak


%\newpage
\bibliographystyle{IEEEbib}
\bibliography{refs}

\newpage
\section{appendix}
\subsection{Proof for Theorem \ref{T1}}
\begin{proof}
\begin{eqnarray*} 
    \|h_{i+1}\|_2 &=& \|x_{i+1}-x^\star\|_2 \\ &=&\|\mathcal{P}_\mathcal{K} (x_{i}-\eta(A^TS^{t^T}S^tAx_{i}-A^TS^{t^T}S^ty))-x^\star\|_2,
\end{eqnarray*}
then because of the distance preservation of translation \cite[Lemma 6.3]{2015_Oymak_Sharp}, we have:
\begin{eqnarray*}            &\stackrel{\text{(a)}}{=}&\|\mathcal{P}_{\mathcal{K}-x^\star} (x_i-x^\star -\eta(A^TS^{t^T}S^tAx^{i}-A^TS^{t^T}S^ty))\|_2 ,
\end{eqnarray*}
then we apply \cite[Lemma 6.4]{2015_Oymak_Sharp}, where $\mathcal{C}$ is the smallest close cone containing the set $\mathcal{K} - x^\star$:
\begin{eqnarray*} 
    &\stackrel{\text{(b)}}{\leq}& \|\mathcal{P}_{\mathcal{C}} (x_i-x^\star -\eta(A^TS^{t^T}S^tAx^{i}-A^TS^{t^T}S^ty))\|_2,
\end{eqnarray*}
\begin{eqnarray*} 
    &=& \|\mathcal{P}_{\mathcal{C}} (h_i +\eta A^TS^{t^T}S^ty-\eta(A^TS^{t^T}S^tA(h_{i}-x^\star))\|_2 \\ 
    &=& \|\mathcal{P}_{\mathcal{C}} ((I-\eta A^TS^{t^T}S^tA)h_i + \eta A^TS^{t^T}S^t(y-Ax^\star)\|_2 \\ 
    &=& \|\mathcal{P}_{\mathcal{C}} ((I-\eta A^TS^{t^T}S^tA)h_i + \eta A^TS^{t^T}S^tw\|_2,
\end{eqnarray*}
then  because of the definition of the cone-projection operator \cite[Lemma 6.2]{2015_Oymak_Sharp} we have:
\begin{eqnarray*} 
    &\stackrel{\text{(c)}}{=}& \sup_{v \in \mathcal{C} \cap \mathcal{B}^d} v^T[(I-\eta A^TS^{t^T}S^tA)h_i +   \eta A^TS^{t^T}S^tw)] \\   
    &\leq& \{\sup_{v \in \mathcal{C} \cap \mathcal{B}^d} v^T(I-\eta A^TS^{t^T}S^tA)h_i \\
    &&+   \eta\sup_{v \in \mathcal{C} \cap \mathcal{B}^d} v^T( A^TS^{t^T}S^tw)) \},
\end{eqnarray*}
next we tidy up the supremums terms by the definition of $\alpha(\eta,S^tA)$ and $\beta(S^t,A)$ in Definition 1:
\begin{eqnarray*} 
    &\leq& \{\sup_{v \in \mathcal{C} \cap \mathcal{B}^d} (v^T(I-\eta A^TS^{t^T}S^tA)\frac{h_i}{\|h_i\|_2} )\|h_i\|_2\\
    &&+   \eta\sup_{v \in \mathcal{C} \cap \mathcal{B}^d} (v^T A^TS^{t^T}S^t\frac{w}{\|w\|_2})\|w\|_2) \}\\ 
    &\leq& \{\sup_{u,v \in \mathcal{C} \cap \mathcal{B}^d} (v^T(I-\eta A^TS^{t^T}S^tA)u )\|h_i\|_2\\
    &&+   \eta\sup_{v \in \mathcal{C} \cap \mathcal{B}^d} (v^T A^TS^{t^T}S^t\frac{w}{\|w\|_2})\|w\|_2) \}\\ 
    &\leq& \{\sup_{u,v \in \mathcal{C} \cap \mathcal{B}^d} (v^T(I-\eta A^TS^{t^T}S^tA)u )\|h_i\|_2\\
    &&+   \eta L\sup_{v \in A\mathcal{C} \cap \mathcal{B}^n} (v^T S^{t^T}S^t\frac{w}{\|w\|_2})\|w\|_2) \}\\  
    &=& \alpha(\eta,S^tA)\|h_i\|_2 + \eta mL\beta(S^t,A)\|w\|_2  
\end{eqnarray*}
%\begin{eqnarray*} 
Then we do recursive substitution:
\begin{eqnarray*} 
    \|h_{i+1}\|_2 &\leq& \alpha^i(\eta,S^tA) \|h_0\|_2 \\
    &&+  m\eta\beta(S^t,A)\frac{1-\alpha^i(\eta,S^tA)}{1-\alpha(\eta,S^tA)}\|w\|_2\\
    &\leq& \alpha^i(\eta,S^tA) \|h_0\|_2 + \frac{m\eta L\beta(S^t,A)}{1-\alpha(\eta,S^tA)}\|w\|_2
\end{eqnarray*} 
\end{proof}

\subsection{Proofs for the explicit bounds}
\subsubsection{Proof for Lemma \ref{L1}}
\begin{proof}
From the results in \cite[Lemma 6.8]{2015_Oymak_Sharp} we can have the following bounds with probability at least $1 - 4e^{-\frac{\theta^2}{8}}$:
\begin{eqnarray*}
    \|SA(u + v)\|_2 &\geq& b_m\|A(u + v)\|_2 - 2\sqrt{\mu_c}(\mathcal{W} + \theta)\\
    &\geq&  \sqrt{\mu_c}(b_m\|u + v\|_2 - 2(\mathcal{W} + \theta)),
\end{eqnarray*}
\begin{eqnarray*}
    \|SA(u - v)\|_2 &\leq& b_m\|A(u - v)\|_2 + 2\sqrt{L}(\mathcal{W} + \theta)\\
    &\leq&  \sqrt{L}(b_m\|u - v\|_2 + 2(\mathcal{W} + \theta)),
\end{eqnarray*}
then we can have:
\begin{eqnarray*}
\alpha &\leq& \frac{1}{4} \left\{ \| u + v\|_2^2 - \eta\|SA(u + v)\|_2^2 - \|u - v\|_2^2 + \eta\|SA(u - v)\|_2^2 \right\} \\
&\leq& \frac{1}{4} \{ \| u + v\|_2^2 - \eta \mu_c(b_m\|u - v\|_2 - 2(\mathcal{W} + \theta))^2 \\
&&- \|u - v\|_2^2 + \eta L(b_m\|u - v\|_2 + 2(\mathcal{W} + \theta))^2 \} \\
&\leq& \frac{1}{4} \{ \| u + v\|_2^2 \\
&& - \eta \mu_c(b_m^2\|u - v\|_2^2 + 4(\mathcal{W} + \theta)^2 -2 b_m(\mathcal{W} + \theta)\|u + v\|_2) \\
&&- \|u - v\|_2^2\\
&&+ \eta L(b_m\|u - v\|_2 + 4(\mathcal{W} + \theta)^2 + 2 b_m(\mathcal{W} + \theta)\|u - v\|_2) \} \\
&\leq& \frac{1}{4}\{ (1 - \eta\mu_c b_m^2)\|u + v\|_2^2\\
&&+ (\eta L b_m^2 - 1)\|u - v\|_2^2\\
&&+ 4 \eta (\mathcal{W} + \theta)^2(L - \mu_c) \\
&&+ 2 \eta b_m(\mathcal{W} + \theta)(L\|u - v\|_2 + \mu_c\|u + v\|_2) \},
\end{eqnarray*}
Now we set $\eta = \frac{1}{b_m^2L}$, since $L \geq \mu_c$, $\|u + v\|_2^2 \leq 4$, $\|u - v\|_2^2 \leq 4$  and $\|u - v\|_2 + \|u + v\|_2 \leq 2\sqrt{2}$, we have:
\begin{eqnarray*}
\alpha &\leq&\frac{1}{4} \{ 4(1 - \frac{\mu_c}{L}) + \frac{4(\mathcal{W} + \theta)^2}{b_m^2 }(1 - \frac{\mu_c}{L}) + \frac{8\sqrt{2}(\mathcal{W} + \theta)}{b_m} \} \\
 &=& (1 - \frac{\mu_c}{L})(1 + \frac{(\mathcal{W} + \theta)^2}{b_m^2}) + \frac{2\sqrt{2}(\mathcal{W} + \theta)}{b_m},
\end{eqnarray*}
with probability at least $1 - 8e^{-\frac{\theta^2}{8}}$. Thus finishes the proof.

\end{proof}

\subsection{The proof for Lemma \ref{L2}}
\begin{proof}
In this proof we first state two inequality from \cite[Lemma 3]{tang2016gradient} which is extented the from \cite[Lemma 6.7]{2015_Oymak_Sharp}:
\begin{equation}
    \|SA(u + v)\|_2 \geq \sqrt{\mu_c}(b_m - \mathcal{W} - \theta) \|u + v\|_2,
\end{equation}
\begin{equation}
    \|SA(u - v)\|_2 \leq \sqrt{L}(b_m + \sqrt{d} + \theta) \|u - v\|_2,
\end{equation}
with probability at least $1 - e^{-\frac{\theta^2}{2}}$, then we can have:
\begin{eqnarray*}
\alpha &\leq& \frac{1}{4} \left\{ \| u + v\|_2^2 - \eta\|SA(u + v)\|_2^2 - \|u - v\|_2^2 + \eta\|SA(u - v)\|_2^2 \right\} \\
&\leq& \frac{1}{4} \{ \| u + v\|_2^2 - \eta\mu_c(b_m - \mathcal{W} - \theta)^2\|u + v\|_2^2 \\
&&- \|u - v\|_2^2 + \eta L(b_m + \sqrt{d} + \theta)^2\|u - v\|_2^2 \} \\
\end{eqnarray*}
Let $\eta = \frac{1}{L(b_m + \sqrt{d} + \theta)^2}$ we have:
\begin{eqnarray*}
\alpha &\leq& \frac{1}{4}\{ 1 - \frac{\mu_c(b_m - \mathcal{W} - \theta)^2}{L(b_m + \sqrt{d} + \theta)^2}\}\|u + v\|_2^2 \\
&\leq& \{ 1 - \frac{\mu_c(b_m - \mathcal{W} - \theta)^2}{L(b_m + \sqrt{d} + \theta)^2}\}
\end{eqnarray*}
with probability at least $(1-2e^{-\frac{\theta^2}{2}})$.
\end{proof}

The proof for Lemma \ref{L3} is almost the same as the proof of \cite[Proposition 2]{tang2016gradient} hence is not needed to be included here.

\end{document}